\documentclass{amsart}[12pt]
\usepackage{mathpazo}
\usepackage{amssymb}

\newtheorem{theorem}{Theorem}[section]
\newtheorem{lemma}[theorem]{Lemma}
\newtheorem*{Acknowledgement}{\textnormal{\textbf{Acknowledgement}}}
\theoremstyle{definition}
\newtheorem{definition}[theorem]{Definition}

\numberwithin{equation}{section}
\newcommand{\beqa}{\begin{eqnarray*}}
	\newcommand{\eeqa}{\end{eqnarray*}}
\newcommand{\beqn}{\begin{eqnarray}}
	\newcommand{\eeqn}{\end{eqnarray}}

\renewcommand{\a}{\alpha}

\newcounter{cnt1}
\newcounter{cnt2}
\newcounter{cnt3}
\newcommand{\blr}{\begin{list}{$($\roman{cnt1}$)$}
		{\usecounter{cnt1} \setlength{\topsep}{0pt}
			\setlength{\itemsep}{0pt}}}
	\newcommand{\bla}{\begin{list}{$($\alph{cnt2}$)$}
			{\usecounter{cnt2} \setlength{\topsep}{0pt}
				\setlength{\itemsep}{0pt}}}
		\newcommand{\bln}{\begin{list}{$($\arabic{cnt3}$)$}
				{\usecounter{cnt3} \setlength{\topsep}{0pt}
					\setlength{\itemsep}{0pt}}}
			\newcommand{\el}{\end{list}}
		\newtheorem{thm}{Theorem}

		\newtheorem{Def}[thm]{Definition}
		
		\newtheorem{rem}[thm]{Remark}
		\newcommand{\Rem}{\begin{rem} \rm}
			\newcommand{\bdfn}{\begin{Def} \rm}
				\newcommand{\edfn}{\end{Def}}

			\title{Nonrough norms in Lipschitz free spaces}  
			\author[ S. Basu , S. Seal ]
			{Sudeshna Basu$^{1}$, Susmita Seal$ ^{2}$ }
			\address{{$^{1}$}   Sudeshna Basu,
				Department of Mathematics and Statistics, 
				Loyola University, 
				Baltimore, MD 21210, USA 
			}
			\email{sudeshnamelody@gmail.com}
			
			\address {{$^{2}$} Susmita Seal, 
				Department of Mathematics, ,
				Ramakrishna Mission Vivekananda Educational and Research Institute , 
				Belur Math,  Howrah 711202,
				West Bengal, India}
			\email{susmitaseal1996@gmail.com}

			\subjclass{46B20, 46B28}
			\keywords{Slices,  Lipschitz free space, Dentable, non-rough norms.}
			\date{}
\begin{document}
\maketitle
\begin{abstract}
In this paper we characterise nonrough norms of Lipschitz free spaces $\mathcal{F}(M)$
 in terms of a new geometric property of the underlying metric space M. 
\end{abstract}

\section{Introduction}
Let $X$ be a {\it real} nontrivial Banach space and $X^*$ denotes the dual of $X$. We denote the closed unit ball by $B_X$, the unit sphere by $S_X$ and the closed ball of radius $r >0$ and center $x$ by $B_X(x, r).$  For bounded subset $C$ of $X,$  slice of $C$ is determined by $$S(C, f, \a) = \{x \in C : f(x) > \mbox{sup}~ f(C) - \a \},$$ where $x^*\in X^*$ and $\a > 0.$ We assume without loss of generality that $\|x^*\| = 1$. Analogously one can define $w^*$-slices in $X^*$. A subset $F$ in $S_{X^*}$ is said to be norming for $X^*$ if $\|x^*\|=\sup_{x\in F} x^*(x)$ for every $x^*\in X^*.$

Given a pointed metric space $M$, i.e., a metric space with a designated origin 0, 
 set  
$$Lip_0(M)=\{f:M\rightarrow \mathbb{R} \ \mathrm{Lipschitz} : f(0)=0\}$$ is a Banach space under the Lipschitz norm $\|f\|=\sup \Big\{ \frac{|f(x)-f(y)|}{d(x,y)} : x\neq y\Big\}.$ For each $m\in M,$ define $\delta_m : Lip_0(M) \rightarrow \mathbb{R}$ by $\delta_m (f) = f(m).$ Then $\delta_m \in Lip_0(M)^*$ with $\|\delta_m\|$ $=d(0,m).$ If we consider $\mathcal{F}(M)=\overline{\mathrm{span}}\{\delta_m : m\in M\}\subset Lip_0(M)^*,$ then it follows that\\ $\mathcal{F}(M)^*= Lip_0(M).$ Also the set $\Big \{ \frac{\delta_x - \delta_y}{d(x,y)}: x\neq y, x,y\in M\Big\}\subset S_{\mathcal{F}(M)}$ is norming for $Lip_0(M)$ and hence $B_{\mathcal{F}(M)}=\overline{co} \Big \{ \frac{\delta_x - \delta_y}{d(x,y)}: x\neq y, x,y\in M\Big\}.$
 The space $\mathcal{F}(M)$ is called Lipschitz-free space over $M.$ 
 If $\mu = \sum_{i=1}^{n} c_i \delta_{x_i}$ where $x_i\in M\setminus \{0\}$ and $c_i \neq 0$ for all $i=1,\ldots,n,$ we will denote supp($\mu$) = $\{x_1,\ldots,,x_n\}.$ Also we call the name molecule for the elements of $\mathcal{F}(M)$ of the form $\frac{\delta_x -\delta_y}{d(x,y)}$ where $x,y\in M$ with $x\neq y.$ 
  For more details on Lipschitz-free spaces see \cite{G}, \cite{W}. 
 
For a Banach space $X,$ we say norm of $X$ is G$\hat{a}$teaux differentiable at $x\in X,$ if
 \begin{equation}\label{gateaux lim}
\lim\limits_{t\rightarrow 0} \frac{\|x+th\|-\|x\|}{t}
\end{equation}
exists for every $h\in X.$ Moreover, if limit ($\ref{gateaux lim}$) exists uniformly for $h\in S_X,$ then norm of $X$ is called Fr$\acute{e}$chet differentiable at $x\in X.$ From Smulyan's Lemma \cite{DGZ} it is known that norm of $X$ is Fr$\acute{e}$chet differentiable at $x$ if and only if $\inf_{\alpha>0} \mathrm{diam} (S(B_{X^*},x,\alpha))=0.$
In \cite{AZ} authors characterised when the norm of $\mathcal{F}(M)$ is Gateaux differentiable at a convex series of molecules $\mu=\sum_n \lambda_n \frac{\delta_{x_n} -\delta_{y_n}}{d(x_n,y_n)}$ with $\|\mu\|=1,$ in terms of geometric conditions on the points $x_n,y_n$ of the underlying metric spaces. Even more they showed that over a uniformly discrete bounded metric space $M,$ norm of $\mathcal{F}(M)$ is Frechet differentiable at a finitely supported elements of $\mathcal{F}(M)$ if and only if norm of $\mathcal{F}(M)$ is Gateaux differentiable there.  

 For $u\in S_X$ 
$\eta(X,u)=\limsup\limits_{\Vert h\Vert\rightarrow 0} \frac{\Vert u+h\Vert+\Vert u-h\Vert-2}{\Vert h\Vert}$ defines the roughness of $X$ at $u.$
	For $\varepsilon>0$, norm of $X$ is said to be
	$\varepsilon$-rough, if $\eta(X,u)\geq \varepsilon$ for every $u\in S_X$. We say that norm of $X$ is rough,
	 if it is $\varepsilon$-rough for some $\varepsilon >0$ and it is non-rough otherwise. It is known that if norm of $X$ is Frechet differentiable at $x\in S_X,$ then norm of $X$ is non-rough. But converse is not true in general. In \cite{JZ} there is an example of Banach space with nonrough norm having no point of Gateaux differentiability. In this paper our aim is to explore non-rough norms of Lipschitz free spaces $\mathcal{F}(M)$ in terms of geometric properties of the underlying metric space $M.$
	 
To prove our aim 
let us recall the following three properties already studied in \cite{Ba}, \cite{BR} and \cite{BS}. 
 
\begin{definition}
A Banach space $X$ has 
\begin{enumerate}
	
	\item  Ball Dentable Property ($BDP$) if $B_X$ has  
	slices of arbitrarily small diameter. 
	
	\item  Ball Huskable Property ($BHP$) if $B_X$ 
	has nonempty relatively weakly open subsets of arbitrarily small diameter.
	
	\item  Ball Small Combination of Slice Property ($BSCSP$) 
	if  $B_X$ has convex combination of slices of arbitrarily 
	small diameter.
\end{enumerate}
\end{definition}
One can analogously define $w^*$-$BDP$, $w^{*}$-$BHP$ and $w^{*}$-$BSCSP$ in a dual space. The implications among all these properties are clearly described in the following diagram. 
$$ BDP \Longrightarrow \quad BHP \Longrightarrow \quad  BSCSP$$ $ \hspace{4.5 cm} \Big \Uparrow \quad \quad\quad\quad\quad \Big \Uparrow \quad \quad\quad\quad\quad \Big \Uparrow$  $$ w^*BDP \Longrightarrow  w^*BHP \Longrightarrow  w^*BSCSP$$

but reverse implications are not true in general \cite{BS}. 
	 It is known that $X$ is non-rough if and only if $X^*$ has $w^*$-BDP \cite{DGZ}, \cite{SBBV}. 
Also from \cite{BGLPRZ} it is known that if an infinite pointed metric space $M$ is unbounded or is not uniformly discrete then any convex combination of slices of $B_{\mathcal{F}(M)^*}$ has diameter two. Therefore, to get the norm of $\mathcal{F}(M)$ is non-rough, we must restrict ourselves on uniformly discrete, bounded pointed metric space. 
In last 4-5 years intensive effort have been done in order to characterise the related concepts like denting points, strongly exposed points, extreme points of $B_{\mathcal{F}(M)}$ in terms of geometric properties of the underlying metric space. For interested readers we refer \cite{AP}, \cite{APPP}, \cite{LPPRZ},\cite{LPRZ}.

\section{main results}
Let us first recall a result from \cite{AZ}.
\begin{lemma}\label{lem az}
\cite[Lemma 2.2]{AZ} Let $I$ be a (finite or infinite) subset of $\mathbb{N}$ and $\beta_{jk}$ (where $j,k\in I$) be real numbers such that $\beta_{jj}=0$ for all $j\in I.$ Then the following are equivalent.
\begin{enumerate}
\item There exist real numbers $\alpha_j$ (where $j\in I$) such that  for all $j,k\in I$
\begin{equation} \label{equ lem az}
\alpha_k\leqslant \alpha_j+\beta_{kj}.
\end{equation}
\item For every finite sequence $i_1,\ldots,i_m$ of indices we have 
\begin{equation}
\beta_{i_1 i_2}+\beta_{i_2 i_3}+\ldots+\beta_{i_{m-1} i_m}+\beta_{i_m i_1}\geqslant 0.
\end{equation}
\end{enumerate}
Moreover, $\alpha_j$ are unique up to additive constant if and only if for every different pair $j,k\in I$ and every $\varepsilon >0$ there exist finite sequence $i_1,\ldots,i_m$ of indices that contain both $j$ and $k$ and we have 
\begin{equation}
\beta_{i_1 i_2}+\beta_{i_2 i_3}+\ldots+\beta_{i_{m-1} i_m}+\beta_{i_m i_1}\leqslant \varepsilon .
\end{equation}
\end{lemma}
By using some techniques from Lemma $\ref{lem az},$ we now prove our next Lemma that gives a general version of Lemma $\ref{lem az}.$
 \begin{lemma}\label{gen az lem}
Let $\varepsilon \geqslant 0,$ $I$ be a (finite or infinite) subset of $\mathbb{N}$ and $\beta_{jk}$ (where $j,k\in I$) be real numbers such that $\beta_{jj}=0$ for all $j\in I.$
Also let there exist atleast one sequence of real numbers $(\alpha_i)_i$  such that $\alpha_k\leqslant \alpha_j+\beta_{kj}$ holds  $\forall j,k\in I.$
 Then the following are equivalent.
  \begin{enumerate}
  \item For any two real sequences $(\alpha '_i)_{i\in I}$ and $(\alpha ''_i)_{i\in I}$ with $\alpha '_k\leqslant \alpha '_j+\beta_{kj}$ and $\alpha ''_k\leqslant \alpha ''_j+\beta_{kj}$ $\forall j,k\in I,$ 
  we have
 \begin{equation}
 |(\alpha '_j - \alpha '_k) -(\alpha ''_j -\alpha ''_k)|\leqslant\varepsilon  \quad \forall j,k\in I
 \end{equation}
\item For every different pair $j,k \in I$ and $\delta >0$ there exist finite sequence of different numbers $i_1,\ldots,i_m$ of indices that contain both $j$ and $k$ and we have 
\begin{equation}
\beta_{i_1 i_2}+\beta_{i_2 i_3}+\ldots+\beta_{i_{m-1} i_m}+\beta_{i_m i_1}< \varepsilon + \delta.
\end{equation}
  \end{enumerate}
 \end{lemma}

\begin{proof}
For every $r,s\in I$ we first define $$B_{rs}= \inf\{\beta_{i_1 i_2}+\beta_{i_2 i_3}+\ldots+\beta_{i_{m-1} i_m} : i_1=r, i_m=s, i_2,\ldots,i_{m-1} \in I, m\in \mathbb{N}\}.$$
In Lemma $\ref{lem az}$, it was shown that 
\begin{equation}
B_{rs}+B_{st}\geqslant B_{rt} \quad \forall r,s,t\in I
\end{equation}
\begin{equation}
  B_{tt}=0 \quad \forall t\in I.
\end{equation}
$(i)\Rightarrow (ii).$ 
Let $j,k\in I$ with $j\neq k$  and $\delta >0.$ Observe that 
\begin{equation}
B_{rk}\leqslant B_{rs} +B_{sk} \leqslant \beta_{rs} + B_{sk} \quad \forall r,s \in I
\end{equation} 
\begin{equation}
-B_{kr} + \beta_{sr}\geqslant -B_{ks}-B_{sr} +\beta_{sr}\geqslant -B_{ks} \quad \forall r,s I
\end{equation}
Therefore by our assumption 
\begin{equation}
|(B_{rk}-B_{sk})-[-B_{kr}- (-B_{ks})]|\leqslant\varepsilon \quad \forall r,s\in I.
\end{equation}
In particular if we put $r=j$ and $s=k,$ then $B_{jk}+B_{kj}\leqslant\varepsilon.$
Also we can choose finite sums 
\begin{equation}\label{1}
\beta_{j i_2}+\beta_{i_2 i_3}+\ldots+\beta_{i_{m-1} k} < B_{jk}+\frac{\delta}{2}
\end{equation}
\begin{equation}\label{2}
\beta_{k i'_2}+\beta_{i'_2 i'_3}+\ldots+\beta_{i'_{l-1} j} < B_{kj}+\frac{\delta}{2}
\end{equation}
Adding $\ref{1}$ and $\ref{2},$ we get 
\begin{equation}
\beta_{j i_2}+\beta_{i_2 i_3}+\ldots+\beta_{i_{m-1} k}
+
\beta_{k i'_2}+\beta_{i'_2 i'_3}+\ldots+\beta_{i'_{l-1} j}
<\varepsilon + \delta.
\end{equation}
$(ii)\Rightarrow (i).$
Consider two real sequences $(\alpha '_i)_{i\in I}$ and $(\alpha ''_i)_{i\in I}$ such that $\alpha '_k\leqslant \alpha '_j+\beta_{kj}$ and $\alpha ''_k\leqslant \alpha ''_j+\beta_{kj}$ $\forall j,k\in I.$ 
 Let us fix $j,k\in I$ and $\delta>0.$ Then by (ii) there exist finite sequence $i_1,\ldots,i_m$ of different indices that contain both $j$ and $k$ such that 
\begin{equation}
\beta_{i_1 i_2}+\beta_{i_2 i_3}+\ldots+\beta_{i_{m-1} i_m}+\beta_{i_m i_1}< \varepsilon + \delta.
\end{equation}
Let $j=i_p$ and $k=i_q$ for some $p,q\in \{1,\ldots,m\}.$
\begin{equation}
B_{jk} + B_{kj} \leqslant (\beta_{i_p i_{p+1}}+\ldots + \beta_{i_{q-1} i_q}) + (\beta_{i_q i_{q+1}}+ \ldots + \beta_{i_{p-1} i_p}) 
< \varepsilon + \delta.
\end{equation}
Observe that $\alpha '_j - \alpha '_k \leqslant B_{jk}$ [see proof of Lemma $\ref{lem az}$]. Similarly $\alpha '_k - \alpha '_j \leqslant B_{kj}< \varepsilon + \delta - B_{jk}.$  Hence 
\begin{equation}
B_{jk} -\varepsilon -\delta < \alpha '_j - \alpha '_k \leqslant B_{jk}
\end{equation}
Similar is also true for $\alpha ''_j$ and $\alpha ''_k.$  Therefore, $|(\alpha '_j - \alpha '_k) -(\alpha ''_j -\alpha ''_k)|<\varepsilon + \delta.$
Since $\delta >0$ is arbitrary, then $|(\alpha '_j - \alpha '_k) -(\alpha ''_j -\alpha ''_k)|\leqslant\varepsilon.$
\end{proof}

\begin{lemma}\label{finite supp}
Let $\varepsilon >0.$ If every $w^*$ slice of $B_{Lip_0(M)}$ of the form 
$S(B_{Lip_0(M)}, \sum_{i=1}^{n} \lambda_i \frac{\delta_{x_i}-\delta_{y_i}}{d(x_i,y_i)} , \alpha)$ with $\|\sum_{i=1}^{n} \lambda_i \frac{\delta_{x_i}-\delta_{y_i}}{d(x_i,y_i)}\|=1$ and $\alpha>0$ $\Big [$where $(x_1,y_1),\ldots,(x_n,y_n)\in M\times M$ such that $x_i\neq y_i$ for all i and $\lambda_1,\ldots,\lambda_n\in \mathbb{R}^+$ such that $\sum_{i=1}^{n} \lambda_i=1\Big]$
 has diameter greater than $\varepsilon,$ then every $w^*$ slice of $B_{Lip_0(M)}$ has diameter greater than $\varepsilon.$
\end{lemma}

\begin{proof}
Consider an arbitrary $w^*$ slice $S(B_{Lip_0(M)}, \mu , \alpha)$ of $B_{Lip_0(M)},$ where $\mu \in S_{\mathcal{F}(M)}$ and $\alpha >0.$ Choose $\varepsilon >0$ such that $\alpha >2\varepsilon.$ Since $B_{\mathcal{F}(M)}=\overline{co} \Big \{ \frac{\delta_x - \delta_y}{d(x,y)}: x\neq y, x,y\in M\Big\},$ then there exist $(x_1,y_1),\ldots,(x_n,y_n)\in M\times M$ with $x_i\neq y_i$ for all i and $\lambda_1,\ldots,\lambda_n\in \mathbb{R}^+$ with $\sum_{i=1}^{n} \lambda_i=1$ such that 
\begin{equation}
\|\mu - \sum_{i=1}^{n} \lambda_i \frac{\delta_{x_i}-\delta_{y_i}}{d(x_i,y_i)}\|<\varepsilon.
\end{equation}
Let $\mu_0=\sum_{i=1}^{n} \lambda_i \frac{\delta_{x_i}-\delta_{y_i}}{d(x_i,y_i)}.$ Then 
\begin{equation}
\|\mu -\frac{\mu_0}{\|\mu_0\|}\|\leqslant \|\mu - \mu_0\|+ \|\mu_0 - \frac{\mu_0}{\|\mu_0\|}\| <\varepsilon + 1-\|\mu_0\| <2\varepsilon.
\end{equation}
Since $\frac{\mu_0}{\|\mu_0\|}\in S_{\mathcal{F}(M)}$ with finite support, then by \cite[Proposition 3.16]{W}  $\frac{\mu_0}{\|\mu_0\|}\in co \Big \{ \frac{\delta_x - \delta_y}{d(x,y)}: x\neq y, x,y\in M\Big\}.$ 
Also it is easy to observe that $S(B_{Lip_0(M)}, \frac{\mu_0}{\|\mu_0\|} , \alpha - 2\varepsilon)\subset S(B_{Lip_0(M)}, \mu , \alpha).$ Hence, diameter of $S(B_{Lip_0(M)}, \frac{\mu_0}{\|\mu_0\|} , \alpha - 2\varepsilon)$ is greater than $\varepsilon$ gives diameter of $S(B_{Lip_0(M)}, \mu , \alpha)$ is greater than $\varepsilon.$
\end{proof}

We require the following result from \cite{AZ} to prove our desired result.
\begin{lemma}\label{az norm 1 lem}
\cite[Theorem 2.4]{AZ}
Let $M$ be a pointed metric space and $\{(x_i,y_i)\}_{i\in I}$ in $(M\times M)\setminus \{(x,x):x\in M\}$ where $I$ is finite or infinite subset of $\mathbb{N}.$ Then $\|\sum_i \lambda_i \frac{\delta_{x_i} - \delta_{y_i}}{d(x_i,y_i)}\|=1$ for some choice of $\lambda_i >0$ with $\sum_i \lambda_i =1$ if and only if for every finite sequence $i_1,\ldots,i_m$ of indices in $I$ we have 
\begin{equation}
\begin{split}
d(x_{i_1},y_{i_1})+ \ldots + d(x_{i_m},y_{i_m})
\leqslant d(x_{i_1},y_{i_2})+\ldots+ d(x_{i_{m-1}},y_{i_m}) +d(x_{i_m},y_{i_1})
\end{split}
\end{equation}
\end{lemma}
Now we are ready to prove the main aim of this paper.
\begin{theorem}
For a uniformly discrete, bounded point metric space $M,$ the following are equivalent.
\begin{enumerate}
\item Norm of $\mathcal{F}(M)$ is nonrough.
\item $Lip_0(M)$ has $w^*$-$BDP.$ 
\item For each $\varepsilon >0$ there exist $\alpha>0$ and $n\in \mathbb{N}$, $\{(x_i,y_i)\}_{i=1}^{n}$ in $(M\times M)\setminus \{(x,x):x\in M\}$ such that
\begin{enumerate}
\item for every pair of different number $j,k\in \{1,\ldots,n\}$ there exist finite sequence of different numbers $i_1,\ldots,i_m\in \{1,\ldots,n\}$ that contains both $j$ and $k$ and 
\begin{equation}\label{th equ 1}
\begin{split}
d(x_{i_1},y_{i_2})+\ldots+ d(x_{i_{m-1}},y_{i_m}) +d(x_{i_m},y_{i_1})-\varepsilon <
d(x_{i_1},y_{i_1})+ \ldots + d(x_{i_m},y_{i_m})
\end{split}
\end{equation}
\item for every sequence $i_1,\ldots,i_m\in \{1,\ldots,n\}$ we have 
\begin{equation}
\begin{split}
d(x_{i_1},y_{i_1})+ \ldots + d(x_{i_m},y_{i_m})
\leqslant d(x_{i_1},y_{i_2})+\ldots+ d(x_{i_{m-1}},y_{i_m}) +d(x_{i_m},y_{i_1})
\end{split}
\end{equation}
\item for every $x\in M$ there exist $s,t\in \{x_1,y_1,\ldots,x_n,y_n\}$ with $s\neq t$ such that 
 \begin{equation} \label{th eq tran}
d(x,s)+d(x,t)<d(s,t)+\varepsilon.
\end{equation}
 and for every 1-Lipschitz map $f:M\rightarrow \mathbb{R}$ with
$$f(x_i)-f(y_i)>d(x_i,y_i) (1-\alpha)\quad \forall i=1,\ldots,n$$ 
we have
\begin{equation}\label{th eq func tran}
f(t)-f(s)>d(s,t)-\varepsilon.
\end{equation}

\end{enumerate} 
\end{enumerate}
\end{theorem}

\begin{proof}
$(i)\Leftrightarrow (ii)$ is straightforward.

Let $c=\inf \{d(x,y): x,y\in M, x\neq y\}$ and $D=$ diam $(M).$

$(ii)\Rightarrow (iii).$ 
Let $\varepsilon >0.$ Then by Lemma $\ref{finite supp}$ there exist $\sum_{i=1}^{n} \lambda_i \frac{\delta_{x_i} - \delta_{y_i}}{d(x_i,y_i)} \in S_{\mathcal{F}(M)}$ (where $\lambda_i >0$ with $\sum_{i=1}^{n} \lambda_i =1$) 
and $\alpha >0$ such that
 $w^*$ slice $S=S(B_{Lip_0(M)},\sum_{i=1}^{n} \lambda_i \frac{\delta_{x_i} - \delta_{y_i}}{d(x_i,y_i)},\alpha)$  of $B_{Lip_0(M)}$ has diameter less than $\frac{\varepsilon}{2D}.$ 
Since $\|\sum_{i=1}^{n} \lambda_i \frac{\delta_{x_i} - \delta_{y_i}}{d(x_i,y_i)}\|=1,$
then condition (b) will follow directly from Lemma $\ref{az norm 1 lem}.$  
  Choose $f\in S_{Lip_0(M)}$ such that $\sum_{i=1}^{n} \lambda_i \frac{f(x_i) - f(y_i)}{d(x_i,y_i)}=1.$ Therefore $\frac{f(x_i) - f(y_i)}{d(x_i,y_i)}=1$ for all $i=1,\ldots,n.$
For each $i,j\in \{1,\ldots,n\}$ we consider $\beta_{ij}=d(x_i,y_j)-d(x_i,y_i)$ and $\alpha_i=f(y_i).$ Observe that 
\begin{equation}
\alpha_i - \alpha_j=f(y_i)-f(y_j)=f(x_i)-d(x_i,y_i)-f(y_j)\leqslant d(x_i,y_j)-d(x_i,y_i)=\beta_{ij}
\end{equation}
Let $(\alpha '_i)_{i=1}^{n}$ be a finite real sequence such that $\alpha '_i - \alpha '_j\leqslant \beta_{ij}.$ Consider $N=\{x_1,y_1,x_2,y_2,\ldots,x_n,y_n\}.$ Define $g:N\rightarrow \mathbb{R}$ as $g(y_i)=\alpha '_i$ and $g(x_i)=\alpha'_i+d(x_i,y_i).$ Then it is easy to check that $g$ a is Lipschitz map. Therefore by \cite[Theorem 1.33]{W} and using a suitable translation, we will get $g_0\in B_{Lip_0(M)}$ with $\sum_{i=1}^{n} \lambda_i \frac{g_0(x_i) - g_0(y_i)}{d(x_i,y_i)}=1$ and $g_0(0)=0.$ Hence $g_0\in S.$ Thus $\|g_0-f\|_{Lip_0(M)}\leqslant \frac{\varepsilon}{2D}.$ In particular, $|(g_0-f)(y_j)-(g_0-f)(y_k)|<\frac{\varepsilon}{2}$ $\forall j,k.$ It gives $|(\alpha_j-\alpha_k)-(\alpha '_j-\alpha '_k)|<\frac{\varepsilon}{2}$ $\forall j,k.$ Hence from Lemma $\ref{gen az lem}$ 
condition (a) follows.
To prove (c), let us first fix $x_0\in M$ and 
we define 
\begin{equation}\notag
\begin{split}
h_1(x)=\inf_{z\in N} [f(z)+d(z,x)]\\
h_2(x)=\sup_{z\in N} [f(z)-d(z,x)]
\end{split}
\end{equation}

Observe that $h_1$ and $h_2$ are 1-Lipschitz map with $h_1|_N=f=h_2|_N.$ Thus
 $\sum_{i=1}^{n} \lambda_i \frac{h_j(x_i) - h_j(y_i)}{d(x_i,y_i)}>1-\alpha$ $\forall j=1,2.$ Consider a suitable translation $\tilde{h_1}$ and $\tilde{h_2}$ of $h_1$ and $h_2$ so that $\tilde{h_1}(0)=0$ and $\tilde{h_2}(0)=0.$ Then $\tilde{h_1},\tilde{h_2}\in S.$
  Choose $s,t\in N$ such that $h_1(x_0)= f(s)+d(x_0,s)$ and $h_2(x_0)=f(t)-d(x_0,t).$ Then
 \begin{equation}\notag
 \begin{split}
 \frac{\varepsilon}{2} > |\tilde{h_1}(x_0) -\tilde{h_1}(s) -(\tilde{h_2}(x_0)-\tilde{h_2}(s))|
 =|h_1(x_0) -h_1(s)-(h_2(x_0)-h_2(s))|\\
 =|h_1(x_0)-h_2(x_0)|\hspace{2.7 cm}\\
 =d(s,x_0)+d(t,x_0)+f(s)-f(t)\hspace{.5 cm}\\
 \geqslant d(s,x_0)+d(t,x_0)-d(s,t)\hspace{1.2 cm}
 \end{split}
 \end{equation}
 Thus, $d(s,x_0)+d(t,x_0)<d(s,t)+\varepsilon.$ Also 
 \begin{equation}\notag
 f(t)-f(s)=h_2(x_0)-h_1(x_0)+d(t,x)+d(s,x)\geqslant-\frac{\varepsilon}{2}+ d(s,t).
 \end{equation}
Consider an 1-Lipschitz map $h:M\rightarrow \mathbb{R}$ with
$h(x_i)-h(y_i)>d(x_i,y_i) (1-\alpha)$ $\forall i=1,\ldots,n.$ Then we can choose a suitable translation $\tilde{h}$ of $h$ such that $\tilde{h}\in S.$ Thus 
 \begin{equation}\notag
 \begin{split}
h(t)-h(s)=\tilde{h}(t)-\tilde{h}(s) 
= f(t)-f(s)+[\tilde{h}(t)-\tilde{h}(s)-(f(t)-f(s))]\\
\geqslant -\frac{\varepsilon}{2}+ d(s,t) - \|\tilde{h}-f\|_{Lip_0(M)} D\hspace{1.7 cm}\\
\geqslant d(s,t)-\varepsilon.\hspace{5 cm}
\end{split}
 \end{equation}

$(iii)\Rightarrow (ii).$  
Let $\varepsilon >0.$ Then by our assumption there exist $0<\alpha <\varepsilon$, $\{(x_i,y_i)\}_{i=1}^{n}$ in $(M\times M)\setminus \{(x,x):x\in M\}$ and $\lambda_1,\ldots,\lambda_n>0$ with $\sum_{i=1}^{n}\lambda_i=1$ (from Lemma $\ref{az norm 1 lem}$) such that $\|\sum_i \lambda_i \frac{\delta_{x_i} - \delta_{y_i}}{d(x_i,y_i)}\|=1.$ Consider $w^*$ slice $T=S(B_{Lip_0(M)},\sum_i \lambda_i \frac{\delta_{x_i} - \delta_{y_i}}{d(x_i,y_i)},\frac{ \alpha  \min \lambda_i}{n })$ of $B_{Lip_0(M)}.$
Define $N=\{x_1,y_1,x_2,y_2,\ldots,x_n,y_n\}.$
Let $f,g\in T.$ 
Observe that 
\begin{equation}
0\leqslant d(x_i,y_i)-(f(x_i)-f(y_i))<\frac{\alpha}{n} d(x_i,y_i)\leqslant \frac{\alpha}{n} D\quad \forall i=1,\ldots,n.
\end{equation}
Similar is true for $g$ and hence 
\begin{equation}\label{equ 1}
|f(x_i)-f(y_i)-(g(x_i)-g(y_i))|<D \frac{\alpha}{n} <D \frac{\varepsilon}{n} \quad \forall i=1,\ldots,n.
\end{equation}
Let $j,k\in \{1,2\ldots,n\}$ with $j\neq k.$ By assumption there exist finite sequence of different numbers $i_1,\ldots,i_m\in \{1,\ldots,n\}$ that contains both $j$ and $k$ and satisfy $\ref{th equ 1}.$ Then
\begin{equation}\notag
\begin{split}
0\leqslant \sum_{t=1}^{m-1} [d(x_{i_t},y_{i_{t+1}})-(f(x_{i_t})-f(y_{i_{t+1}}))]\hspace{5 cm}\\
<\varepsilon-d(x_{i_m},y_{i_1})+\sum_{t=1}^{m} d(x_{i_t},y_{i_t})-\sum_{t=1}^{m-1}(f(x_{i_t})-f(y_{i_{t+1}}))\hspace{2.5 cm}\\
=\varepsilon-d(x_{i_m},y_{i_1})+\sum_{t=1}^{m} d(x_{i_t},y_{i_t})-\sum_{t=1}^{m}(f(x_{i_t})-f(y_{i_t})) -f(y_{i_1})+f(x_{i_m})\\
<\varepsilon-d(x_{i_m},y_{i_1}) +m \frac{\alpha}{n} D -f(y_{i_1})+f(x_{i_m})\hspace{4.5 cm}\\
\leqslant\varepsilon-d(x_{i_m},y_{i_1}) + \alpha D +d(x_{i_m},y_{i_1}) \hspace{5.3 cm}\\
=\varepsilon + \alpha D\hspace{9.2 cm}\\
<\varepsilon (1+D)\hspace{9 cm}
\end{split}
\end{equation}
Thus we have
\begin{equation}
0\leqslant \sum_{t=r}^{s} [d(x_{i_t},y_{i_{t+1}})-(f(x_{i_t})-f(y_{i_{t+1}}))]<\varepsilon (1+D) \quad \forall r\leqslant s,\ r,s=1,\ldots,m-1.
\end{equation}
Similar is also true for $g$ and hence 
\begin{equation}\label{equ 2}
|\sum_{t=r}^{s} [f(x_{i_t})-f(y_{i_{t+1}})-(g(x_{i_t})-g(y_{i_{t+1}}))]|<\varepsilon (1+D) \quad \forall r\leqslant s,\ r,s=1,\ldots,m-1.
\end{equation}
Without loss of generality we can assume that $j=i_p$, $k=i_q$ and $p> q$ for some $p,q\in \{1,\ldots,m\}.$
By using $\ref{equ 1}$ and $\ref{equ 2}$ we have 
\begin{equation}\label{sp equ 1}
\begin{split}
|(f-g)(x_j)-(f-g)(x_k)|
\leqslant |\sum_{t=q+1}^{p} (f-g)(x_{i_t})-(f-g)(y_{i_t})|\\
+|\sum_{t=q+1}^{p} (f-g)(y_{i_t})-(f-g)(x_{i_{t-1}})|\\
<n \frac{D\varepsilon}{n}+\varepsilon (1+D)\hspace{2.7 cm}\\
=\varepsilon + 2D\varepsilon \hspace{4 cm}
\end{split}
\end{equation}
\begin{equation}\label{sp equ 2}
\begin{split}
|(f-g)(y_j)-(f-g)(y_k)|
\leqslant \sum_{t=q+1}^{p}|(f-g)(y_{i_t})-(f-g)(x_{i_{t-1}})|\\
+\sum_{t=q+1}^{p}|(f-g)(x_{i_{t-1}})-(f-g)(y_{i_{t-1}})|\\
<\varepsilon (1+D) + n \frac{D\varepsilon}{n}\hspace{2.7 cm}\\
=\varepsilon + 2D\varepsilon \hspace{4 cm}
\end{split}
\end{equation}
\begin{equation}\label{sp equ 3}
\begin{split}
|(f-g)(x_j)-(f-g)(y_k)|
\leqslant \sum_{t=q}^{p}|(f-g)(x_{i_t})-(f-g)(y_{i_t})|\\
+\sum_{t=q+1}^{p}|(f-g)(y_{i_t})-(f-g)(x_{i_{t-1}})|\\
<n \frac{D\varepsilon}{n}+\varepsilon (1+D)\hspace{2.7 cm}\\
=\varepsilon + 2D\varepsilon \hspace{4 cm}
\end{split}
\end{equation}
\begin{equation}\label{sp equ 4}
\begin{split}
|(f-g)(y_j)-(f-g)(x_k)|
\leqslant \sum_{t=q}^{p}|(f-g)(y_{i_t})-(f-g)(x_{i_{t-1}})|\\
+\sum_{t=q+1}^{p}|(f-g)(x_{i_{t-1}})-(f-g)(y_{i_{t-1}})|\\
<\varepsilon (1+D) + n \frac{D\varepsilon}{n}\hspace{2.7 cm}\\
=\varepsilon + 2D\varepsilon \hspace{4 cm}
\end{split}
\end{equation}

Thus from $\ref{equ 1}$, $\ref{sp equ 1}$, $\ref{sp equ 2}$, $\ref{sp equ 3}$ and $\ref{sp equ 4}$
we can conclude that
\begin{equation}
|(f-g)(u)-(f-g)(v)|<\varepsilon (1+2D)\quad \forall u,v\in N
\end{equation}

Let $x,y\in M$ such that $x\neq y.$ Then for $x\in M$ there exist $s,t\in N$ with $s\neq t$ satisfy $\ref{th eq tran}$ and $\ref{th eq func tran}.$ Observe that
\begin{equation}\notag
\begin{split}
\frac{d(x,t)}{d(x,s)+d(x,t)} \frac{f(t)-f(x)}{d(x,t)} + \frac{d(x,s)}{d(x,s)+d(x,t)} \frac{f(x)-f(s)}{d(x,s)}\\
=\frac{f(t)-f(s)}{d(x,s)+d(x,t)} \hspace{6.5 cm}\\
>\frac{d(s,t)-\varepsilon}{d(s,t)+\varepsilon}\hspace{7.5 cm}\\
\geqslant 1-\frac{2\varepsilon}{c+\varepsilon}\hspace{7.5 cm}.
\end{split}
\end{equation}
Thus 
\begin{equation}
1\geqslant\frac{f(t)-f(x)}{d(x,t)}  > 1-\frac{d(x,s)+d(x,t)}{\min\{d(x,s),d(x,t)\}} \frac{2\varepsilon}{c+\varepsilon} \geqslant 1-\frac{2D}{c} \frac{2\varepsilon}{c+\varepsilon}
\end{equation}
Similarly is also true for $g$ and hence 
\begin{equation}
|\frac{f(t)-f(x)}{d(x,t)} -\frac{g(t)-g(x)}{d(x,t)}|\leqslant \frac{4 D \varepsilon}{c(c+\varepsilon)}
\end{equation}
Therefore, 
\begin{equation}
|f(t)-f(x) -(g(t)-g(x))|\leqslant \frac{4 D^2 \varepsilon}{c(c+\varepsilon)}
\end{equation}
\end{proof}
Similarly for $y\in M$ there exists $t'\in N$ such that 
\begin{equation}
|f(t')-f(y) -(g(t')-g(y))|\leqslant \frac{4 D^2 \varepsilon}{c(c+\varepsilon)}
\end{equation}
Finally, 
\begin{equation}\notag
\begin{split}
\frac{|(f-g)(x)-(f-g)(y)|}{d(x,y)}
\leqslant \frac{1}{c} [|(f-g)(x)-(f-g)(t)|+|(f-g)(t)-(f-g)(t')|\\
+|(f-g)(t')-(f-g)(y)|]\\
\leqslant \frac{8 D^2 \varepsilon}{c(c+\varepsilon)} + \varepsilon (1+2D).\hspace{5.2 cm}
\end{split}
\end{equation}
Since $f,g$ are any arbitrary choosen element from $S,$ hence $S$ has diameter less than $\frac{8 D^2 \varepsilon}{c(c+\varepsilon)} + \varepsilon (1+2D).$ Therefore $Lip_0(M)$ has $w^*$-$BDP.$

  \begin{Acknowledgement}
The second  author's research is funded by the National Board for Higher Mathematics (NBHM), Department of Atomic Energy (DAE), Government of India, Ref No: 0203/11/2019-R$\&$D-II/9249.
\end{Acknowledgement}

\end{document}